\documentclass{amsart}

\usepackage{amsmath, amssymb, amsthm, marginnote, ltxcmds, adjustbox, stmaryrd, graphicx, bbold, extarrows}
\usepackage{esvect}

\usepackage{mathtools, stackengine, changepage, ragged2e}

\usepackage{todonotes}

\usepackage{mathrsfs}

\definecolor{cite}{HTML}{11871E}
\definecolor{url}{HTML}{698996}
\definecolor{link}{HTML}{912F1B}

\usepackage[pdfencoding=unicode, colorlinks=true, linkcolor=link, citecolor=cite, urlcolor=url, linktocpage]{hyperref}

\usepackage[backend=biber, style=alphabetic, maxnames=10, maxalphanames=3, minalphanames=3]{biblatex}
\usepackage{tikz}
\usetikzlibrary{cd}
\usetikzlibrary{arrows, arrows.meta, positioning, calc}
\tikzcdset{arrow style=tikz, diagrams={>={Straight Barb[scale=0.8]}}}

\tikzstyle{arrow} = [-{Straight Barb[scale=0.8]}, line width=0.2mm]
\tikzset{
	math to/.tip={Glyph[glyph math command=rightarrow]},
	loop/.tip={Glyph[glyph math command=looparrowleft, swap]},
}

\usepackage{enumitem}
	{\end{enumerate}%
}

\usepackage[
letterpaper,
twoside=false,
textheight=22cm,
textwidth=14.4cm,
marginparsep=0.75cm,
marginparwidth=2.5cm,
heightrounded,
centering
]{geometry}

\usepackage{lineno}

\usepackage{contour}
\usepackage{ulem}

\contourlength{0.5pt}

\newcommand{\myuline}[1]{%
	\uline{\phantom{#1}}%
	\llap{\contour{white}{#1}}%
}

\makeatletter
\newcommand*{\saved@myuline}{}
\let\saved@myuline\myuline

\newcommand*{\mathuline}{%
	\mathpalette{\math@myuline\saved@myuline}%
}
\newcommand*{\math@myuline}[3]{%
	\mbox{#1{$#2#3\m@th$}}%
}

\renewcommand*{\myuline}{%
	\relax  
	\ifmmode
	\expandafter\mathuline
	\else
	\expandafter\saved@myuline
	\fi
}
\makeatother

\usepackage{libertine}
\usepackage{stmaryrd}

\DeclareFontFamily{T1}{cbgreek}{}
\DeclareFontShape{T1}{cbgreek}{m}{n}{<-6>  grmn0500 <6-7> grmn0600 <7-8> grmn0700 <8-9> grmn0800 <9-10> grmn0900 <10-12> grmn1000 <12-17> grmn1200 <17-> grmn1728}{}
\DeclareSymbolFont{quadratics}{T1}{cbgreek}{m}{n}
\DeclareMathSymbol{\qoppa}{\mathord}{quadratics}{21}

\usepackage[capitalise]{cleveref}

\Crefname{prop}{Proposition}{Propositions}
\Crefname{lem}{Lemma}{Lemmas}
\Crefname{cor}{Corollary}{Corollaries}
\Crefname{thm}{Theorem}{Theorems}
\Crefname{alphThm}{Theorem}{Theorems}

\Crefname{defn}{Definition}{Definitions}
\Crefname{notation}{Notation}{Notations}
\Crefname{cons}{Construction}{Constructions}
\Crefname{rmk}{Remark}{Remarks}
\Crefname{obs}{Observation}{Observations}
\Crefname{trick}{Trick}{Tricks}
\Crefname{warning}{Warning}{Warnings}
\Crefname{conj}{Conjecture}{Conjectures}
\Crefname{assump}{Assumption}{Assumptions}
\Crefname{recollect}{Recollection}{Recollections}
\Crefname{terminology}{Terminology}{Terminologies}

\Crefname{question}{Question}{Questions}
\Crefname{example}{Example}{Examples}

\Crefname{figure}{Figure}{Figures}

\crefformat{equation}{(#2#1#3)}
\crefformat{section}{\S#2#1#3}
\crefmultiformat{section}{\S\S#2#1#3}{and~#2#1#3}{, #2#1#3}{, and~#2#1#3}

\newtheorem{thm}[subsection]{Theorem}
\newtheorem{prop}[subsection]{Proposition}
\newtheorem{lem}[subsection]{Lemma}
\newtheorem{cor}[subsection]{Corollary}

\newtheorem{alphThm}{Theorem}

\newcommand{\neutralize}[1]{\expandafter\let\csname c@#1\endcsname\count@}
\makeatother

\newtheorem*{thm*}{Theorem}
\newtheorem*{prop*}{Proposition}
\newtheorem*{lem*}{Lemma}
\newtheorem*{cor*}{Corollary}

\newtheorem{alphConj}{Conjecture}



\newtheorem{alphCor}{Corrollary}



\newtheorem{alphProp}{Proposition}



\theoremstyle{definition}
\newtheorem*{defn*}{Definition}
\newtheorem{defn}[subsection]{Definition}

\newtheorem{nota}[subsection]{Notation}

\newtheorem{question}[subsection]{Question}

\theoremstyle{remark}
\newtheorem{rmk}[subsection]{Remark}
\newtheorem{obs}[subsection]{Observation}

\newtheorem{constr}[subsection]{Construction}

\newcommand{\family}{\mathcal{F}}

\newcommand{\spc}{\mathrm{An}}

\newcommand{\sC}{{\mathcal C}}

\DeclareMathOperator{\presheaf}{\mathcal{P}}
\newcommand{\orbit}{\mathcal{O}}

\newcommand{\finite}{\mathrm{Fin}}

\DeclareMathOperator{\fib}{\operatorname{fib}}

\newcommand{\hAut}{\mathrm{hAut}}
\newcommand{\Asph}{\mathrm{Asph}}
\newcommand{\Asphinj}{\mathrm{Asph}^{\mathrm{inj}}}
\newcommand{\Asphinjfam}[1]{\mathrm{Asph}^{\mathrm{inj},{#1}}}

\newcommand{\abstrsp}{\mathrm{Trans}}
\newcommand{\trsp}[1]{\mathrm{Trans}_{/B{#1}}}
\newcommand{\Bor}{\mathrm{Bor}}

\newcommand{\univspc}[2]{E_{#1} {#2}}
\newcommand{\mcg}[1]{\mathrm{Mod}({#1})}

\newcommand{\Out}[1]{\mathrm{Out}({#1})}
\newcommand{\Aut}[1]{\mathrm{Aut}({#1})}

\newcommand{\pCH}{\mathrm{pCH}}



\title{Borel actions in nonpositively curved geometry and the Nielsen realisation problem}
\author{ \large\textsc{Christian} KREMER}

\date{\today}

\begin{document}
	\maketitle
	
		\begin{abstract}
		In this note, we record the proof of a theorem about the coincidence of genuine and homotopy fixed points for isometric group actions on complete Riemannian manifolds with nonpositive sectional curvature, and more generally, certain quotients of universal spaces for families. The result is put into context with the Nielsen realisation problem for aspherical manifolds, and we give a unifying account of different formulations of that problem, made possible by the same methods.
	\end{abstract}
	
	\tableofcontents
	
	\section{Introduction}
	
	Let $M$ be a complete Riemannian manifold with nonpositive sectional curvature. Assume the finite group $G$ acts on $M$ via isometries. For each subgroup $H \subset G$, the fixed submanifold $M^H \subset M$ is an embedded submanifold, which in fact inherits a Riemannian metric with nonpositive sectional curvature as well. This submanifold will be referred to as the \textit{genuine $H$-fixed points} of the action. This article is about the following homotopical property of such actions, comparing the genuine fixed points of $M$ to the homotopy fixed points.
	
	\begin{thm}
		\label{thm:actions_non_nonpositively_curved_manifolds_are_borel}
		If $M$ is a complete Riemannian manifold with nonpositive sectional curvature and an isometric action by a finite group $G$, then for each $H \leq G$ the canonical map
		\[ M^H \rightarrow M^{hH} \]
		from the genuine $H$-fixed points of $M$ to the homotopy $H$-fixed points, is an equivalence.
	\end{thm}
	
	Recall that a $G$-space $X$ is a \textit{Borel $G$-space} if for each $H \leq G$ the map $X^{H} \rightarrow X^{hH}$ is an equivalence. So \cref{thm:actions_non_nonpositively_curved_manifolds_are_borel} predicts that nonpositively curved manifolds with isometric $G$-action are Borel $G$-spaces. 
	In this note, we record a proof of \cref{thm:actions_non_nonpositively_curved_manifolds_are_borel} in terms of genuine equivariant homotopy theory, depending on the Cartan-Hadamard theorem from differential geometry as geometric input. The proof is achieved by translating - using said theorem from geometry - to a result about certain quotients of universal spaces for families. That result is \cref{thm:universal_spaces_and_their_borel_quotients}, and should in fact be viewed as the main mathematical content of this note. We briefly explain \cref{thm:universal_spaces_and_their_borel_quotients} and put it into context with the Nielsen realisation problem after that.
	
	Recall that for a discrete group $\Gamma$, a \textit{family of subgroups} is a collection $\family$ of subgroups closed under conjugation and passing to subgroups. Given such a family $\family$, a \textit{universal $\Gamma$-space for the family $\family$} is a $\Gamma$-space $\univspc{\family}{\Gamma}$ with the property that for each subgroup $H\leq \Gamma$ one has
	\begin{equation}
		(\univspc{\family}{\Gamma})^H \simeq \begin{cases}
			* \colon \text{if $H \in \family$};\\
			\emptyset \colon \text{if $H \notin \family$. }
		\end{cases}
	\end{equation}
	A particularly important example of a family for this note is the family $\finite$ of  \textit{finite subgroups} of a discrete group $\Gamma$, because of the following simple consequence of the Cartan-Hadaramard theorem.
	
	\begin{obs}[Proved as \cref{prop:cartan_hadamard_with_action_models_quotient_of_universal_space}]
		\label{obs:cartan_hadamard}
		Let $M$ be a connected closed Riemannian manifold with nonpositive sectional curvature with an isometric $G$-action. For some choice of base point $x \in M$, let $\pi=\pi_1(M,x)$ and $\Gamma = \pi_1(M_{hG},x)$. Then there is a $G$-homotopy equivalence 
		\begin{equation}
			M \simeq  \pi \backslash \univspc{\finite}{\Gamma}.
		\end{equation}
	\end{obs}
	
	In \cref{obs:cartan_hadamard}, we note that the map $M \rightarrow M_{hG}$ lets us view $\pi$ as a normal subgroup of $\Gamma$ with $G = \Gamma / \pi$, so we can view the quotient $\pi \backslash \univspc{\finite}{\Gamma}$ as a $G$-space with its residual action.
	Using \cref{obs:cartan_hadamard}, we can deduce \cref{thm:actions_non_nonpositively_curved_manifolds_are_borel} from the following purely homotopy theoretic statement.
	
	\begin{thm}
	\label{thm:universal_spaces_and_their_borel_quotients}
		Let $\Gamma$ be a discrete group, let $\pi \leq \Gamma$ be a normal subgroup, and $\family$ the family of subgroups $K \subset \Gamma$ such that $K \leq \Gamma \rightarrow \Gamma/\pi$ is injective.  Then, for each $H \leq G$, the map
		\[ (\pi \backslash \univspc{\family}{\Gamma})^{H} \rightarrow (\pi \backslash \univspc{\family}{\Gamma})^{hH}   \]
		is an equivalence.
	\end{thm}	

    \begin{proof}[Proof of \cref{thm:actions_non_nonpositively_curved_manifolds_are_borel} assuming \cref{thm:universal_spaces_and_their_borel_quotients} and \cref{obs:cartan_hadamard}]
        We apply \cref{thm:universal_spaces_and_their_borel_quotients} for the family of finite subgroups of the group $\Gamma$ from \cref{obs:cartan_hadamard}. Since $\pi_1(M)$ is torsion free \cite[Cor. 12.18]{Lee}, the finite subgroups of $\Gamma$ are precisely those $F \subset \Gamma$ such that $F \subset \Gamma \rightarrow \Gamma/\pi_1(M) \simeq G$ is injective. So indeed, $M \simeq \pi_1(M) \backslash \univspc{\finite}{\Gamma}$ is a Borel $G$-space.
    \end{proof}

	\subsection*{Relevance for the Nielsen realisation problem}
	
	The Nielsen realisation problem for aspherical manifolds\footnote{For the purposes of this article, we follow the convention that an aspherical space is a connected space with contractible universal cover, and an aspherical
	manifold is a closed connected manifold with contractible universal cover.} asks if homotopical group actions on aspherical manifolds, in a sense to be made precise below, can be rigidified to actual group actions. 
	Nielsen originally asked if for a finite subgroup $G$ of the mapping class group $\mcg{\Sigma}$ of a closed oriented surface $\Sigma$, each element of $G$ admits a representative homeomorphism, which jointly assemble to an actual $G$-action on $M$. A positive answer was ultimately given by Kerckhoff \cite{Kerckhoff}. The map $\mcg{\Sigma} \rightarrow \pi_0 \hAut(\Sigma)$ to the group of homotopy classes of homotopy automorphisms is an isomorphism, so one may indeed view this as a way of rigidifying a homotopical group action.

    The latter perspective can be generalised to general aspherical manifolds - can a group homomorphism $G \rightarrow \pi_0 \hAut(M)$ be rigidified to an actual group action?  This question has received quite a bit of attention in the literature \cite{dl5, Weinberger_variations}. If the centre of $\pi_1(M)$ is not trivial, its formulation is not yet quite optimal, and we recall some homotopy theory of aspherical spaces to give a slightly better formulation.
    For a general aspherical space $X$, its space of homotopy automorphisms $\hAut(X)$ has the homotopy groups
	\begin{equation}
		\pi_i\hAut(X) = \begin{cases}
			\Out{\pi_1(X)}  &\colon\text{for $i = 0$;}\\
			C(\pi_1(X))  &\colon\text{for $i=1$;}\\
			0  &\colon\text{for $i\geq 2$.}
		\end{cases}
	\end{equation}
	Here $\Out{\pi_1(M)}$ is the group of \textit{outer automorphisms} of $\pi_1(M)$, the quotient $\Aut{\pi_1(M)}/G$, where $G$ acts on itself by conjugation, and $C(\pi_1(X))$ is the centre of $\pi_1(X)$.	
	
	In \cite{RaymondScott}, Raymond-Scott observe that if $G \rightarrow \pi_0 \hAut(M)$ lifts to a group action on $M$, one can in particular construct the fibre sequence of aspherical spaces\footnote{They construct a short exact sequence of groups, which amounts to the same datum under the Bar construction $B(-)$.} $M \rightarrow M_{hG} \rightarrow BG$. 
    The outer automorphism of $\pi_1(M)$ induced by fibre transport along $g \in \pi_1(BG) = G$ is induced by conjugating with any lift of $g$ in $\pi_1(M_{hG})$ and noting that this automorphism preserves the subgroup $\pi_1(M) \leq \pi_1(M_{hG})$.
	
	By the straightening-unstraightening equivalence, constructing such a fibre sequence is the same as finding a lift in the following diagram of $E_1$-groups in the category of spaces.
	 \begin{equation}
	 	\begin{tikzcd}
	 		& \hAut(M) \ar[d] \\
	 		G \ar[r] \ar[ur, dashed] & \pi_0 \hAut(M)
	 	\end{tikzcd}
	 \end{equation}
	 Raymond-Scott observe that if $C(\pi_1(M))$ is nontrivial, such a lift need not exist, and in fact give explicit aspherical manifolds for which it does not. It seems to make sense to view Raymond-Scott's objection as a reason to change the formulation of the Nielsen realisation problem to the following.
	 
	 \begin{question}[The Nielsen realisation problem]
        \label{quest:nielsen}
	 	Let $G$ be a finite group, $M$ an aspherical manifold and consider a map of $E_1$-groups $G \rightarrow \hAut(M)$. Does it refine to a $G$-action on $M$?\footnote{To avoid confusion, we remark that this is not the same as finding a lift along the map of $E_1$-groups $\mathrm{Homeo}(M) \rightarrow \hAut(M)$, unless $\mathrm{Homeo}(M) $ is equipped with the \textit{discrete} topology.}
	 \end{question}
	 
	 For brevity, we will refer to a map of $E_1$-groups $G \rightarrow \hAut(M)$ as a \textit{homotopical $G$-action}.
	 There are numerous known cases in which the Nielsen realisation problem holds true, and we review two of them with connection to negative curvature.
	 \begin{enumerate}
	 	\item If $M = \Sigma_g$ is an oriented surface of genus $g \geq 2$, then there exists a hyperbolic metric on $M$ so that the homotopical action by $G$ refines to an action by isometries \cite{Kerckhoff}.
	 	\item If $M$ is a closed hyperbolic manifold of dimension $d \geq 3$, then the isometry group of $M$ is a finite group and the map $\mathrm{Isom}(M) \rightarrow \hAut(M)$ is an equivalence, by Mostow rigidity.
	 \end{enumerate}
	 Note that in both cases, \cref{thm:actions_non_nonpositively_curved_manifolds_are_borel} applies. This is good news, since it means that for each subgroup $H \leq G$, the homotopy type of the fixed point space $M^H$ is aleady determined by the map of $E_1$-groups $G \rightarrow \hAut(M)$, since they are equivalent to the homotopy fixed points $M^{hH}$.
	 This makes it reasonable to ask the following variant of the Nielsen realisation problem.
	 
	 \begin{question}[Borel version of the Nielsen realisation problem]
	 \label{quest:borel_version_of_nielsen}
	 	Let $G$	be a finite group, $M$ an aspherical manifold and $G \rightarrow \hAut(M)$ a map of $E_1$-groups. Does it refine to a Borel $G$-action on $M$?
	 \end{question}
	 
    Essentially since the spaces $M^{hH}$ are determined by the homotopical group action of $G$ on $M$, we note that this formulation of the Nielsen realisation problem is about the existence of a group action giving rise to a specific $G$-homotopy type, the \textit{Borelification} $\mathrm{Bor}(M)$. Here the Borelification of a space $X$ with a homotopical $G$-action is the $G$-space $\mathrm{Bor}(X)$ with $X^H = X^{hH}$, which is constructed using the right adjoint to the forgetful functor $\spc_G \rightarrow \spc^{BG}$.  While it seems like a harder problem to construct a Borel $G$-action on $M$ as opposed to an arbitrary one, in practice it is convenient that it gives a genuine $G$-homotopy type to work with. In fact, all solutions to \cref{quest:nielsen} known to the author even solve \cref{quest:borel_version_of_nielsen}.
	Note that through \cref{thm:universal_spaces_and_their_borel_quotients}, \cref{quest:borel_version_of_nielsen} is closely related to the following question.
	
	\begin{question}[Manifold models for universal spaces]
		\label{quest:borel_version_of_nielsen_cocompact_version}
		Let $\Gamma$ be a discrete group which admits a subgroup $\pi \leq \Gamma$ of finite index, which is a Poincar\'e duality group. Is there a cocompact $\Gamma$-manifold $N$ which, as $\Gamma$-space is equivalent to $\univspc{\finite}{\Gamma}$?
	\end{question}

	To pass between \cref{quest:borel_version_of_nielsen} and \cref{quest:borel_version_of_nielsen_cocompact_version} one either passes to a $\pi$-cover or takes the quotient by the $\pi$-action. Uniqueness up to $\Gamma$-homeomorphism of manifold models for $\univspc{\finite}{\Gamma}$ is what is often referred to as the \textit{equivariant Borel conjecture}, adding a second layer of reasoning to calling \cref{quest:borel_version_of_nielsen} the Borel version of \cref{quest:nielsen}.

    We warn the reader that \cref{quest:borel_version_of_nielsen} and \cref{quest:nielsen} are known to be false in full generality \cite{BlockWeinberger}, but widely open if one focuses on groups $G$ of odd order. The mechanism by which it fails for the group $C_2$ is rather well understood and has to do with the $L$-theory of group rings of the infinite dihedral group, and under the absence of $2$-torsion this mechanism will not make an appearance, see \cite[Thm. 7.2.(2)]{dl5} for an instance of this.
	 
	 \subsection*{Relation to other work}
	 
	 The content of this note is, in one formulation or another, known to experts on the Nielsen realisation problem, although we are not aware of a written record. Having been asked questions about the relation of \cref{quest:nielsen} and \cref{quest:borel_version_of_nielsen_cocompact_version} several times, we write it to have a unifying and clarifying account of the homotopy theoretic basics of the Nielsen realisation problem, relevant to the author's ongoing collaborations on that topic \cite{pd2, isovariantpd}. 
	 
	 \subsection*{Conventions}
	 
	 We work in the setting of $\infty$-categories as developed by Joyal, Lurie and many others. The term \textit{category} will refer to a an $\infty$-category. We write $\spc$ for the category of spaces. As justified by Elmendorf's theorem, for a (discrete) group $\Gamma$ we define the category $\spc_\Gamma$ of $\Gamma$-spaces as the category of $\spc$-valued presheaves on the orbit category $\orbit(\Gamma)$.
	 
	 \subsection*{Acknowledgements}
	 
	 The author expresses gratitude to Andrea Bianchi, Emma Brink, Wolfgang L\"uck and Shmuel Weinberger for helpful conversations around this note.
	 
	 \section{Quotients of universal spaces for families}
	 
	 The purpose of this section is to prove \cref{thm:universal_spaces_and_their_borel_quotients} about the equivariant homotopy theory of spaces of the form $\pi \backslash \univspc{\family}{\Gamma}$.
	 
	\begin{defn}
		Let $\Asph \subset \spc$ denote the full subcategory of aspherical  spaces.
		Further, write $\Asphinj \hookrightarrow \Asph$ for the wide subcategory on maps that induce injections on fundamental groups for arbitrary choices of basepoints.
	\end{defn}
	
	\begin{constr}
		If $\Gamma$ is a discrete group, and $\Gamma/H \in \orbit(\Gamma)$ a transitive $\Gamma$-set, then its \textit{transport groupoid} $\abstrsp(\Gamma/H) \coloneqq (\Gamma/H)_{h\Gamma}$, defined as the homotopy orbits by the $\Gamma$-action, is an aspherical space, which identifies with $BH$. The $\Gamma$-equivariant map $\Gamma/H \rightarrow \Gamma/\Gamma$ induces a map $\abstrsp(\Gamma/H) \rightarrow \abstrsp(\Gamma/\Gamma) \simeq B\Gamma$. So, applying $\abstrsp$ upgrades to a functor
		\begin{equation}
			\label{eq:transport_to_aspherical_spaces}
			\trsp{\Gamma} \colon \orbit_\family(\Gamma) \xrightarrow{\simeq} \Asphinjfam{\family}_{/B\Gamma}, \hspace{3mm} G/H \mapsto [(\Gamma/H)_{hG} = BH \rightarrow B\Gamma].
		\end{equation}
	\end{constr}
	
	\begin{nota}
		If $\family$ is a family of subgroups of the discrete group $\Gamma$, we will write $\Asphinjfam{\family}_{/B\Gamma}$ for the full subcategory on those $f \colon X \rightarrow B\Gamma$, for which the image of $\pi_1(X,x) \subset \pi_1(B\Gamma,f(x)) \cong \Gamma$ is in $\family$, where the last equivalence is given by conjugation with some path in $B\Gamma$ from $f(x)$ to the basepoint. Note that this does not depend on the choice of base point, since families are by definition invariant under conjugation.
	\end{nota}
	
	\begin{obs}
		\label{obs:orbit_categories_and_aspherical_spaces}
		For every family $\family$ of subgroups of $\Gamma$, the functor $\trsp{\Gamma}$ from \cref{eq:transport_to_aspherical_spaces} induces an equivalence of categories $\orbit_{\family}(\Gamma) \xrightarrow{\sim} \Asphinjfam{\Gamma}_{/B\Gamma}$. This is a simple consequence of covering theory: the category $\Asphinj_{/B\Gamma}$ is by definition the category of connected coverings of $B\Gamma$, so the functor $\trsp{\Gamma}$ gives an equivalence to the category of transitive $\Gamma$-sets.
		An explicit inverse sends $p \colon X \rightarrow B\Gamma$ to the fibre over $* \in B\Gamma$. It is easy to see that this equivalence respects families as claimed.
	\end{obs}

	\begin{lem}
		\label{lem:extension_family_and_aspherical_spaces}
		Consider an extension of groups
		\[ 1 \rightarrow \pi \rightarrow \Gamma \rightarrow G \rightarrow 1 \]
		and write $\family$ for the family of subgroups $F \subset \Gamma$ such that the composite $F \subset \Gamma \rightarrow G$ is injective. Then the diagram
		\begin{equation*}
			\begin{tikzcd}
				\Asphinjfam{\family}_{/B\Gamma} \ar[r, "\trsp{\Gamma}"] \ar[d] & \spc_{/B\Gamma} \ar[d]\\
				\Asphinj_{/BG} \ar[r, "\trsp{G}"] & \spc_{/BG}
			\end{tikzcd}
		\end{equation*}
		where the vertical maps are induced by postcomposition with $B\Gamma \rightarrow BG$, is cartesian.
	\end{lem}
	
	\begin{proof}
		Both horizontal functors are inclusions of full subcategories. We have to show that $X \rightarrow B\Gamma \in \spc_{/B\Gamma}$ is an object of $\Asphinjfam{\family}_{/B\Gamma}$ if and only if $X \rightarrow B\Gamma \rightarrow BG$ is an object of $\Asphinj_{/BG}$. Asphericity of $X$ is of course necessary in both situations, so the claim is that $\pi_1(X) \rightarrow \pi_1(B\Gamma)$ is injective with image in $\family$ if and only if $\pi_1(X) \rightarrow \pi_1(B\Gamma) \rightarrow \pi_1(BG)$ is injective, which is exactly the definition of the family $\family$.
	\end{proof}
	
	\begin{lem}
		\label{lem:spaces_over_borelification}
		Let $X \in \spc^{BG}$ be a space with $G$-action. Then the diagram
		\begin{equation*}
			\begin{tikzcd}
				\spc^{BG}_{/X} \ar[r, "\Bor"] \ar[d] & (\spc_G)_{/\Bor(X)} \ar[d]\\
				\spc^{BG} \ar[r, "\Bor"] & \spc_G
			\end{tikzcd}
		\end{equation*}
		is cartesian.
	\end{lem}
	
	\begin{proof}
		Since $\Bor \colon \spc^{BG} \rightarrow \spc_G$ is fully faithful, both horizontal functors are fully faithful. Now $Z \rightarrow \Bor(X)$ lies in the image of the upper horizontal map if and only if $Z$ is in the image of the lower horizontal map, establishing the claim.
	\end{proof}
	
	\begin{obs}
		\label{obs:residual_action_borel}
		In an extension of groups
		\[ 1 \rightarrow \pi \rightarrow \Gamma \rightarrow G \rightarrow 1 \]
		note that $B\pi$ attains a $G$-action. Indeed, the object $B\Gamma \rightarrow BG \in \spc_{/BG}$ describes, under straightening-unstraightening, a space with $G$-action with underlying space $\fib(B\Gamma \rightarrow BG) = BG$.
	\end{obs}
	
	\begin{thm}
		\label{thm:extensions_and_universal_spaces}
		Consider an extension of groups
		\[ 1 \rightarrow \pi \rightarrow \Gamma \rightarrow G \rightarrow 1 \]
		and write $\family$ for the family of subgroups $F \subset \Gamma$ such that the composite $F \subset \Gamma \rightarrow G$ is injective. Then there is an equivalence of categories
		\[ \spc_\Gamma^\family \simeq  (\spc_G)_{/\Bor(B\pi)} \]
		which sends $X \in \spc_\Gamma^\family$ to $\pi \backslash X$, where the map $\pi \backslash X \rightarrow \Bor(B\pi)$ comes from the action of $\pi$ on $X$ being free, so that we have the $G$-equivariant map $(\pi \backslash X)^e \simeq X_{h\pi} \rightarrow *_{h\pi} \simeq B\pi$.
	\end{thm}
	
	\begin{proof}
		Recall that if $\sC$ is a small category, and $X \in \presheaf(\sC)$ is a presheaf, then $\presheaf(\sC)_{/X} \simeq \presheaf(\sC_{/X})$, where $\sC_{/X} = \sC \times_{\presheaf(\sC)} \presheaf(\sC)_{/X}$. Applying this to $(\spc_{G})_{/\Bor(B\pi)}$ we get an equivalence
		\[ (\spc_{G})_{/\Bor(B\pi)} \simeq \presheaf(\orbit(G)  \times_{\spc_G} (\spc_{G})_{/\Bor(B\pi)} ). \]
		Since $\spc_{\Gamma}^\family \simeq \presheaf(\orbit_\family(\Gamma))$, we have reduced to provide an equivalence $\orbit_\family(\Gamma) \simeq \orbit(G)  \times_{\spc_G} (\spc_{G})_{/\Bor(B\pi)} $. 
		The claim follows from the following commuting diagram of cartesian squares.
		\begin{equation}
			\begin{tikzcd}
				\orbit_\family(\Gamma) \ar[r] \ar[d, , "\pi \backslash - "] & \Asphinjfam{\family}_{/B\Gamma} \ar[r] \ar[d] & \spc_{/B\Gamma} \simeq (\spc^{BG})_{/B\pi} \ar[d] \ar[r, "\Bor"] & (\spc_G)_{/\Bor(B\pi)} \ar[d] \\
				\orbit(G) \ar[r] & \Asphinj_{/BG} \ar[r] & \spc_{/BG} \ar[r, "\Bor"] & \spc_G
			\end{tikzcd}
		\end{equation}
		The first square is consists of equivalences in the horizontal direction as in \cref{obs:orbit_categories_and_aspherical_spaces}. The second square is cartesian as a consequence of \cref{lem:extension_family_and_aspherical_spaces}, and the last square by \cref{lem:spaces_over_borelification}.
	\end{proof}
	
	\begin{cor}
		In the situation of \cref{thm:extensions_and_universal_spaces} is a $G$-equivariant equivalence
		\[ \pi \backslash E_\family \Gamma \simeq \Bor(B\pi). \]
	\end{cor}
	
	\begin{proof}
		The equivalence of categories in \cref{thm:extensions_and_universal_spaces} has to send the terminal object of $\spc_\Gamma^\family$, that is $E_\family \Gamma$, to the terminal object of $(\spc_G)_{/\Bor(B\pi)}$, which is clearly $\Bor(B\pi)$. But the explicit description of that equivalence shows that it is implemented by taking the quotient by the $\pi$-action, leading to the assertion.
	\end{proof}
	
	\section{Actions in nonpositive curvature}
	
	In the following we will abbreviate
	\[ \pCH = \text{complete Riemannian manifold of nonpositive sectional curvature} \]
	invoking ``pseudo Cartan-Hadamard". In the literature, a Cartan-Hadamard manifold is a $\pCH$-manifold which is additionally simply connected, hence the notation.
	
	\begin{thm}[Cartan-Hadamard]
		Let $M$ be a $\pCH$, and let $x \in M$. Then the exponential map $T_x M \rightarrow M$ is a covering map. In particular, if $M$ is additionally simply connected, it is contractible.
	\end{thm}

	\begin{lem}
		Let $M$ be a $\pCH$ with an isometric action by the finite group $G$. Then $M^G$, with the restricted Riemannian metric, is a $\pCH$ as well.
	\end{lem}
	
	\begin{proof}
		Completeness of $M^G$ follows from $M^G \subset M$ being a closed subset, using that a closed subspace of a complete metric space space is complete again. Further, $M^G \subset M$ is easily seen to be totally geodesic, which implies that the sectional curvature of $M^G$ is the restriction of the sectional curvature of $M$.
	\end{proof}
	
	\begin{prop}
		\label{prop:cartan_hadamard_with_action_models_quotient_of_universal_space}
		Let $M$ be a closed connected $\pCH$ with an isometric action by the finite group $G$. Let $\Gamma = \pi_1(M_{hG})$. Then $M \simeq \pi_1(M) \backslash \univspc{\finite}{\Gamma}$,
	\end{prop}
	
	\begin{proof}
		Let $\widetilde{M}$ denote the universal cover of $M$. Let $\widetilde{\Gamma}$ denote the group of all homeomorphisms  $h$ of $\widetilde{M}$ for which there is a $g \in G$ and a commuting diagram as follows.
		\begin{equation}
			\begin{tikzcd}
				\widetilde{M} \ar[r, "h"] \ar[d] & \widetilde{M} \ar[d]\\
				M \ar[r, "g"] & M 
			\end{tikzcd}
		\end{equation}
		In particular, we get an inclusion $\pi_1(M) \hookrightarrow \widetilde{\Gamma}$ as the deck transformation group, and $G = \widetilde{\Gamma} / \pi_1(M)$. The residual $G$-action on $M = \pi_1(M) \backslash \widetilde{M}$ is exactly the $G$-action we started with. Hence, since $\widetilde{M}$ is contractible
		\[ B\widetilde{\Gamma} \simeq \widetilde{M}_{h\widetilde{\Gamma}} \simeq M_{hG} \]
		giving an identification $\widetilde{\Gamma} \simeq \Gamma$. Now $\widetilde{M}$ is a simply connected $\pCH$ with an isometric action by the group $\widetilde{\Gamma}$. In particular every finite subgroup of $\widetilde{\Gamma}$ has contractible fixed points: the fixed point set by a finite subgroup is geodesically convex, and nonempty by Cartan's fixed point theorem \cite[Thm. 12.17]{Lee}. Let us now argue that every infinite subgroup $H \leq \widetilde{\Gamma}$ contains a nontrivial element of the deck transformation group, so that it in particular has no fixed points. This follows since the intersection of $H$ with the kernel of $\widetilde{\Gamma} \rightarrow G$ is still infinite, in particular contains a nontrivial element of that kernel, which is precisely the deck transformation group. Hence $\widetilde{M} \simeq \univspc{\finite}{\widetilde{\Gamma}}$, and by construction of $\widetilde{\Gamma}$ there is a $G$-equivariant equivalence $M \simeq \pi_1(M) \backslash \widetilde{M}$.
	\end{proof}

    \begin{rmk}
        It seems likely that the results presented in this section also hold for other nonpositively curved metric spaces, which also enjoy a Cartan-Hadamard theorem \cite{Bridson_Haefliger}.
    \end{rmk}
	
	\printbibliography
	
\end{document}